\documentclass{ifacconf}

\newcommand{\R}{{\mathbb R}}

\newcommand{\exb}[1]{{\mathbb E}\left[ #1 \right]}   

\newcommand{\rmd}{{\rm d}}

\newcommand{\tr}{\mathop{\rm Trace}\nolimits}

\newcommand{\Rd}{\R ^{d}}

\newcommand{\calC}{{\mathcal C}}

\newcommand{\calM}{{\mathcal M}}
\newcommand{\calN}{{\mathcal N}}

\newcommand{\bbE}{{\mathbb E}}

\newcommand{\KL}[2]{D_{\mathrm{KL}}(#1 \parallel #2)}

\newcommand{\subj}{\mathrm{subject \,to}}

\usepackage{todonotes}
\usepackage{graphicx}      
\usepackage{natbib}      
\usepackage{mathtools}
\usepackage{xcolor}
\usepackage{enumitem}
\usepackage{amsmath}
 \usepackage{amssymb}
 \usepackage{amsfonts}
 \usepackage{algorithm}
\usepackage{algpseudocode}
\usepackage{tabularray}

\makeatletter
\newcommand{\firstpageheadline}[1]{%
  \gdef\@firstpageheadline{#1}%
}
\gdef\@firstpageheadline{}

\let\old@logohead\@logohead
\def\@logohead{%
\bgroup
\footnotesize
\savebox{\@tempboxa}{\vtop {%
\hbox to \textwidth{\hfill \@firstpageheadline \hfill}%
\hbox to \textwidth{}%
\hbox to \textwidth{}%
\hbox to \textwidth{}%
\hbox to \textwidth{}%
}}%
\box\@tempboxa
\egroup
}
\makeatother

\begin{document}
\firstpageheadline{%
\parbox{0.95\textwidth}{\centering\large\bfseries
To appear in the Proceedings of MTNS 2026 (extended abstracts).\\
Submitted on February 15, 2026; accepted on April 20, 2026.
}%
}
\begin{frontmatter}

\title{Unbalanced Optimal Transport and Density Control for Discrete-Time Linear Systems}

\thanks[footnoteinfo]{This work was supported in part by Moonshot R\&D Grant Number JPMJMS2021 and by JSPS KAKENHI under Grant Number JP21H04875.}

\author[First]{Haruto Nakashima} 
\author[Second]{Siddhartha Ganguly} 
\author[Third]{Kenji Kashima}

\address[First]{Kyoto University, 
   Kyoto, Sakyo-Ward Yoshida Honmach Japan (e-mail: nakashima.haruto.72v@st.kyoto-u.ac.jp).}
\address[Second]{Georgia Institute of Technology, 
   Daniel Guggenheim School of Aerospace Engineering, 270 Ferst Dr., Atlanta, USA (e-mail: sganguly41@gatech.edu).}
\address[Third]{Kyoto University, 
   Kyoto, Sakyo-Ward Yoshida Honmach Japan (e-mail: kk@kyoto-u.ac.jp).}

\begin{abstract}                
This article studies unbalanced optimal transport (UOT) and its dynamical extension, unbalanced density control (UDC), for a class of constrained discrete-time linear systems. UOT compares measures with unequal total mass by balancing transport cost and fidelity to reference measures, while UDC incorporates system dynamics and constraints into this framework. Focusing on Gaussian references and discrete-time linear systems, we show that both problems admit globally optimal convex formulations, analogous to covariance steering. A numerical experiment is provided to illustrate our approach.
\end{abstract}


\end{frontmatter}

\section{Introduction}\label{section:intro}

Optimal transport (OT) measures the discrepancy between two distributions (more generally, measures) via the minimum cost of transporting mass, and has been widely used in machine learning, control, and related fields; see, e.g., \cite{Peyre25,Chen2021}.
However, classical OT requires equal total mass and thus cannot directly handle mismatched mass caused by missing data or dissipation. Unbalanced optimal transport (UOT) relaxes this restriction by penalizing deviations from given reference measures (e.g., by KL divergence), enabling comparison of nonnegative measures with different masses (\cite{Sejourne2023},~\cite{Payre2018},~\cite{Wang2024}).

As standard OT theory is translated in the setting where distribution evolves according to underlying dynamics (often discussed under covariance steering; see \cite{Liu2025,Nakashima25,Morimoto2024}), we study \emph{unbalanced density control} (UDC) as a discrete-time linear-systems extension of UOT problem. 
We focus on Gaussian reference measures and show that both UOT and UDC admit globally optimal solutions via convex optimization. UOT reduces to finite-dimensional optimization over mean and covariance with closed-form mass computation, and UDC admits an SDP formulation for fixed mass, combined with an analogous closed-form mass update.

\noindent\textbf{Contributions:}
\begin{itemize}
    \item We derive a finite-dimensional convex formulation and a globally optimal algorithm for Gaussian UOT.
    \item We formulate Gaussian UDC for discrete-time linear systems and propose a globally optimal SDP-based algorithm.
\end{itemize}
A \emph{significantly expanded version} of this paper, containing additional theoretical results, complete proofs, and numerical experiments, is available in~\cite{ref:UnOTDenCon:Extended:TAC}.


\section{Preliminary}\label{section:Preliminary}
We denote as \(\R\) the set of real numbers. Let \(p \ge 1\); over \(\Rd\), we denote the set of nonnegative finite measures with finite $p$-th moments by $\calM_p^{+}(\Rd)$. For $f \in \calM_p^{+}(\Rd)$, we call
$\int \rmd f$ the \emph{mass} of $f$. The space of integrable, nonnegative functions over \(\Rd\) will be denoted by \(L^{1}_{+}(\Rd)\). We denote by $\calN(m,\Sigma)$ the Gaussian probability measure with mean $m$ and covariance $\Sigma \succeq O$ (the null matrix). For a nonnegative scalar $c>0$, a measure of the form
\[
c\,\calN(m,\Sigma)
\]
is called a \emph{Gaussian measure}. For absolutely continuous measures $\alpha, \beta \in  \calM_p^{+}(X)$ 
the standard Kullback--Leibler (KL) divergence is defined by:
\begin{equation*}\label{eq:KL_def}
    \KL{\alpha}{\beta} \coloneqq \int_{\mathbb{R}^d} \log \left(\frac{\rmd\alpha}{\rmd\beta}(x)\right) \rmd\alpha(x) - \int_{\Rd} \rmd\beta + \int_{\Rd} \rmd\alpha .
\end{equation*}

\section{Problem Setup}\label{sec:problem}
The first component of our UDC framework is an \emph{unbalanced} variant of the classical optimal transport problem, which we briefly review below.
\subsection{Unbalanced Optimal Transport Problem}\label{subsec:uot}
Fix a positive integer \(d\). Given Gaussian measures
\[
\alpha = c_\alpha \calN(m_\alpha,\Sigma_\alpha), \qquad
\beta  = c_\beta  \calN(m_\beta,\Sigma_\beta),
\]
where \(c_{\alpha},c_{\beta}>0\), \(m_{\alpha},m_{\beta} \in \Rd\), and \(\Sigma_{\alpha},\,\Sigma_{\beta}\succ O\). A transport plan (or coupling) between \(\alpha\) and \(\beta\) will be any nonnegative finite measure \(\pi \in \calM_2^{+}(\Rd \times \Rd)\) which, for simplicity, we assume to be absolutely continuous i.e., we focus on couplings of the form \(\pi(\rmd x_1, \rmd x_2) = \pi(x_1,x_2)\,\rmd x_1 \rmd x_2,\) with \(\pi\in L^1_+(\mathbb R^{2d})\) and having finite second moment. Consequently, we have the marginal conditions
\[
\pi_1(x_1) \coloneqq \int_{\Rd} \pi(x_1,x_2)\,\rmd x_2, \quad \pi_2(x_2) \coloneqq \int_{\Rd} \pi(x_1,x_2)\,\rmd x_1.
\]
Let \(\gamma>0\) and the ground cost \(\ell(x_1,x_2)\coloneqq \|x_2-x_1\|^2\); with these ingredients, we consider the unbalanced optimal transport problem:
\begin{equation}\label{eq:uot}
\begin{aligned}
\inf_{\pi} 
\int_{\Rd}\int_{\Rd} & \|x_2-x_1\|^2 \,\rmd \pi(x_1,x_2) \\
&+ \gamma \KL{\pi_1}{\alpha}
+ \gamma \KL{\pi_2}{\beta}.
\end{aligned}
\end{equation}
Note that this is equivalent to the following formulation where the marginals are treated as explicit decision variables:
\begin{equation}\label{eq:uot_marginals}
\begin{aligned}
\inf_{\pi,\pi_1,\pi_2}\quad &
\int_{\Rd}\int_{\Rd}  \|x_2-x_1\|^2 \,\rmd \pi(x_1,x_2) \\
& + \gamma \KL{\pi_1}{\alpha} + \gamma \KL{\pi_2}{\beta}\\
\subj\quad & \pi_1(x_1) = \int_{\Rd} \pi(x_1,x_2)\,\rmd x_2,\\
& \pi_2(x_2) = \int_{\Rd} \pi(x_1,x_2)\,\rmd x_1 .
\end{aligned}
\end{equation}
The measure $\pi$ is also referred to as a \emph{transport plan}.

\subsection{Unbalanced Density Control Problem}\label{subsection:udc}
 We consider the case where the distribution evolves according to dynamics in~\eqref{eq:uot}.
In this paper, we study the following optimal density control problem with a quadratic control cost and soft endpoint constraints via KL divergence:
\begin{equation}\label{eq:udc}
    \begin{aligned}
        \inf_{\pi_1,\,\{U_k\}_{k=1}^{T-1}} \quad &\exb{\sum_{k=1}^{T-1}\|u_k\|^2} + \gamma \KL{\pi_1}{\alpha} \\
        & \qquad\qquad\qquad\,\,+ \gamma \KL{\pi_T}{\beta}\\
        \subj\quad & x_{k+1} = A x_k + B u_k,\\
        & x_k \sim \pi_k,\qquad k=1,2,\ldots,T-1,\\
        & u_k \sim U_k(\cdot\,|\,x)\ \text{given } x_k=x .
    \end{aligned}
\end{equation}
Here, the expectation in the objective is understood as:
\[
\exb{\sum_{k=1}^{T-1}\|u_k\|^2} =
\sum_{k=1}^{T-1}\int_{\Rd}\!\int_{\Rd} \|u\|^2 \,\rmd U_k(u|x)\,\rmd \pi_k(x),
\]
and we emphasize that $\pi_k$ is not necessarily a probability measure.

\section{Global Solution Method to UOT}\label{section:uot-solution}

\subsection{Restriction to Gaussian marginals}
\label{subsection:uot-gaussian-restriction}

In this section, we prove that there exists an optimal solution whose marginals \(\pi_1,\,\pi_2\) are Gaussian measures.
Let
\begin{align*}
    J(\pi):= \int_{\Rd}\int_{\Rd} &\|x_2-x_1\|^2 \,\rmd \pi(x_1,x_2) \\
    &+ \gamma \KL{\pi_1}{\alpha} + \gamma \KL{\pi_2}{\beta}.
\end{align*}

We first recall the following lower bound for the quadratic transport cost.

\begin{prop}[Transport cost lower bound]\label{prop:transport-lb}
Let $\pi_1,\pi_2\in\calM_2^{+}(\Rd)$  with the same mass $c_1>0$.
Let the mean and covariance of the normalized measures $\pi_i/c_1$ be $m_i$ and $\Sigma_i$, respectively.
Define \(\calC\) as follows:
\begin{equation*}
    \calC:=\left\{
        \begin{aligned}
            \pi:&\int \rmd\pi =c_1, \\
            &\frac{\pi_{i}}{c_{1}}\mbox{ has mean }m_i, \mbox{covariance }\Sigma_{i}, \\
            &\mbox{for } i=1,2.
        \end{aligned} \right\}
\end{equation*}
Then, it holds that:
\begin{equation}\label{eq:gelbrich}
    \begin{aligned}
        &\inf_{\pi\in\calC} \int_{\Rd}\int_{\Rd} \|x_2-x_1\|^2 \,\rmd \pi(x_1,x_2) \\ 
        &\ge c_1\Big(\|m_2-m_1\|^2 \\
        &\quad+\tr\left(\Sigma_1+\Sigma_2-2(\Sigma_1^{1/2}\Sigma_2\Sigma_1^{1/2})^{1/2}\right)\Big).
    \end{aligned}
\end{equation}
Moreover, if $\pi_1$ and $\pi_2$ are Gaussian measures, equality holds and an optimal coupling is induced by an affine map:
\begin{align*}\label{eq:uot-affine-map}
    &x_2 = T x_1 + t, \quad T = \Sigma_1^{-1/2}(\Sigma_1^{1/2}\Sigma_2\Sigma_1^{1/2})^{1/2}\Sigma_1^{-1/2},\\
    &t = m_2 - Tm_1.
\end{align*}
\end{prop}
This proposition states that under fixed mean and covariance, ~\eqref{eq:gelbrich} gives minimal cost for transportation, and the equality holds when the marginal \(\pi_1\) and \(\pi_2\) are Gaussian.
Similarly to~\eqref{prop:gauss-min-kl}, the following proposition formalizes the optimality of Gaussian measure for KL term in \(J(\pi)\).

\begin{prop}[Gaussian minimizes KL]
\label{prop:gauss-min-kl}
Fix $c>0$, $m,m'\in\Rd$, and $\Sigma,\Sigma'\in\mathbb{R}^{d \times d}$ with $\Sigma \succeq O$ and $\Sigma'\succ O$.
For all measures $\pi$ having mass $c$ and whose normalized measure $\pi/c$ has mean $m$ and covariance $\Sigma$, it holds that:
\begin{align}
    \KL{\pi}{c'\calN(m',\Sigma')}\geq\KL{c\calN(m,\Sigma)}{c'\calN(m',\Sigma')}
\end{align}
namely, the Gaussian measure $c\,\calN(m,\Sigma)$ minimizes  \\
$\KL{\cdot}{c'\calN(m',\Sigma')}$ under fixed mass, mean, and covariance.
\end{prop}

Combining Propositions~\ref{prop:transport-lb} and~\ref{prop:gauss-min-kl}, we obtain the following key reduction.

\begin{thm}[Optimality of Gaussian.]
\label{thm:gaussian-optimality}
Let $(\pi_1,\pi_2,\pi)$ be any feasible solution to~\eqref{eq:uot_marginals} with common mass $c>0$.
Let $(m_i,\Sigma_i)$ be the mean and covariance of $\pi_i/c$ $(i=1,2)$, and define Gaussian measures:
\begin{align}
    \pi_i^{\rm G} := c\,\calN(m_i,\Sigma_i), \qquad i=1,2.
\end{align}
If $\alpha$ and $\beta$ are Gaussian measures, then there exists a feasible Gaussian solution $\pi^{\rm G}$
(with marginals $\pi_1^{\rm G},\pi_2^{\rm G}$ and an affine optimal coupling) such that
\begin{align}
    J(\pi^{\rm G}) \le J(\pi).
\end{align}
\end{thm}
As a result of Theorem~\ref{thm:gaussian-optimality}, the problem~\eqref{eq:uot} can be, without loss of optimality, reduced to optimization over the mass, mean, and covariance of marginals.

\subsection{Separation of the optimization}\label{subsection:uot-separation}
Using the closed-form expressions of the Wasserstein-2 distance and the KL divergence between Gaussian measures, the UOT problem~\eqref{eq:uot} is equivalent to the following finite-dimensional optimization:
\begin{align}\label{eq:finite-dim-uot}
    \inf_{c,m_1,\Sigma_1,m_2,\Sigma_2} \quad c\,M(m_1,m_2) + c\,C(\Sigma_1,\Sigma_2) + \psi(c),
\end{align}
where
\begin{align}\label{eq:def-of-psi}
    &M(m_1,m_2) := \|m_2-m_1\|^2  \nonumber\\
    &\qquad\qquad\qquad +\frac{\gamma}{2}(m_1-m_\alpha)^\top\Sigma_\alpha^{-1}(m_1-m_\alpha)  \nonumber\\
    &\qquad\qquad\qquad +\frac{\gamma}{2}(m_2-m_\beta )^\top\Sigma_\beta ^{-1}(m_2-m_\beta )   \nonumber\\
    &C(\Sigma_1,\Sigma_2) :=-2\tr\left(\sqrt{\Sigma_1^{1/2}\Sigma_2\Sigma_1^{1/2}}\right) \nonumber\\
    &\qquad\qquad\quad + \frac{\gamma}{2}\tr(\Sigma_\beta^{-1}\Sigma_{2}) -\frac{\gamma}{2}\log\det\Sigma_{2}+\tr\Sigma_{2}  \nonumber\\
    &\qquad\qquad\quad + \frac{\gamma}{2}\tr(\Sigma_\alpha^{-1}\Sigma_1)-\frac{\gamma}{2}\log\det\Sigma_1+\tr\Sigma_{1} \nonumber\\
    &\psi(c):= \frac{c\gamma}{2}(\log\det\Sigma_{\alpha} + \log\det\Sigma_{\beta} -2d)   \nonumber\\
    &\qquad\quad + \gamma\phi_{\alpha}(c) + \gamma\phi_{\beta}(c) \\
    &\phi_\nu(c):=c\log\!\frac{c}{c_\nu}-c+c_\nu, \quad \nu = \alpha,\ \beta\nonumber
\end{align}

Although~\eqref{eq:finite-dim-uot} is not jointly convex, Algorithm 1 solves the problem globally by exploiting its structure. The algorithm implement  optimization over $c$ and $(m_1,\Sigma_1,m_2,\Sigma_2)$ separately. To introduce the algorithm, define the convex subproblem that is solved in the algorithm:
\begin{align}\label{eq:uot-subproblem}
    \min_{m_1,\Sigma_1,m_2,\Sigma_2}\quad M(m_1,m_2) + C(\Sigma_1,\Sigma_2).
\end{align}

\begin{thm}[Global optimality of Algorithm~\ref{alg:uot}]\label{thm:uot-global}
Algorithm~\ref{alg:uot} yields a globally optimal solution to the UOT problem~\eqref{eq:uot}.
\end{thm}

\begin{pf}
Let $p^\ast$ denote the optimal value of~\eqref{eq:uot-subproblem}. Since
\begin{align}
&c\,M(m_1,m_2) + c\,C(\Sigma_1,\Sigma_2) + \psi(c) \\
= &c\big(M(m_1,m_2)+C(\Sigma_1,\Sigma_2)\big) + \psi(c),
\end{align}
we have
\begin{align}
\inf_{c,m_1,\Sigma_1,m_2,\Sigma_2}
\big\{c\,M + c\,C + \psi(c)\big\}
=
\inf_{c>0}\ \Big\{ c\,p^\ast + \psi(c)\Big\}.
\label{eq:separation}
\end{align}
The right-hand side is a one-dimensional convex optimization in $c$; hence the first-order optimality condition
gives the unique minimizer in closed form, which is exactly~\eqref{eq:c-star}. \qed
\end{pf}

Since the subproblem~\eqref{eq:uot-subproblem} is convex, it can be solved globally via convex programming solver. Using the result of it, Algorithm~\ref{alg:uot} computes the optimal mass \(c^*\) in the second step.


\begin{algorithm}[h]
\caption{UOT between Gaussians}
\label{alg:uot}
\begin{enumerate}
\item \textbf{Mean and covariance optimization (convex).}
Solve~\eqref{eq:uot-subproblem} \\
\(p^*\longleftarrow\) Optimal value of~\eqref{eq:uot-subproblem}
\item \textbf{Mass computation (closed form).}
\begin{align}\label{eq:c-star}
    c^\ast \longleftarrow &\sqrt{c_\alpha c_\beta} \exp\left(-\frac{p^\ast}{2\gamma} - \frac{L}{4} \right), \\
    &(L := \log\det(\Sigma_\alpha)+\log\det(\Sigma_\beta)-2d.) \nonumber
\end{align}
\end{enumerate}
\end{algorithm}

\begin{figure*}[h]\label{fig:optimal-plan}
  \vspace{-10mm}
  \hspace{15mm}
  \includegraphics[scale=0.27]{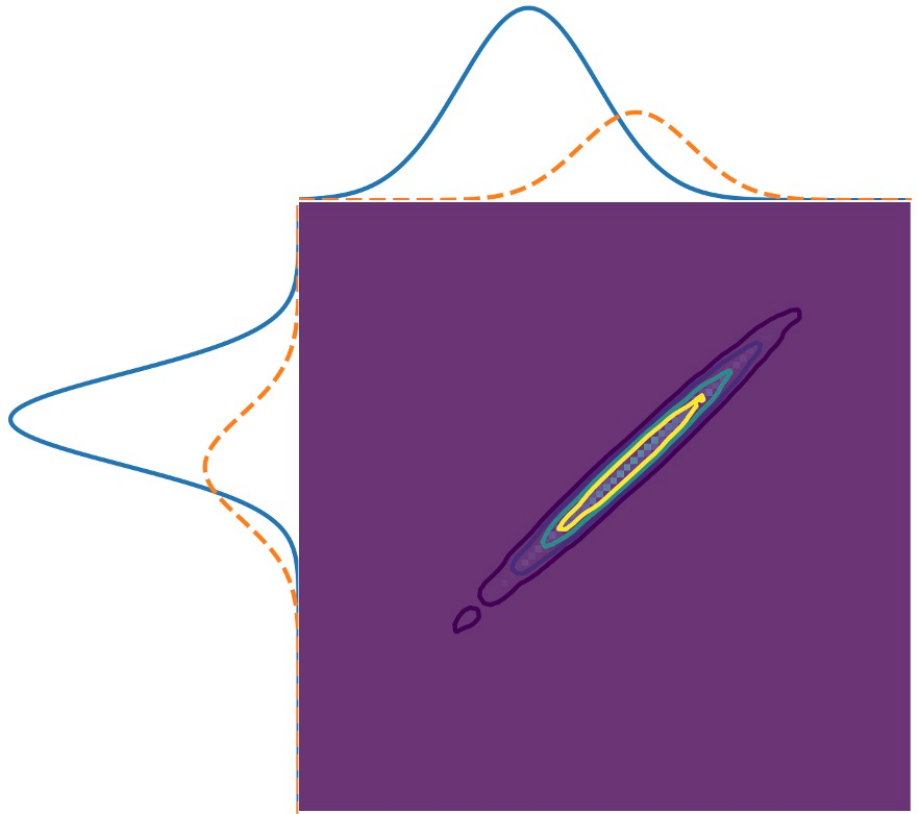}
  \hspace{5mm}
  \includegraphics[scale=0.27]{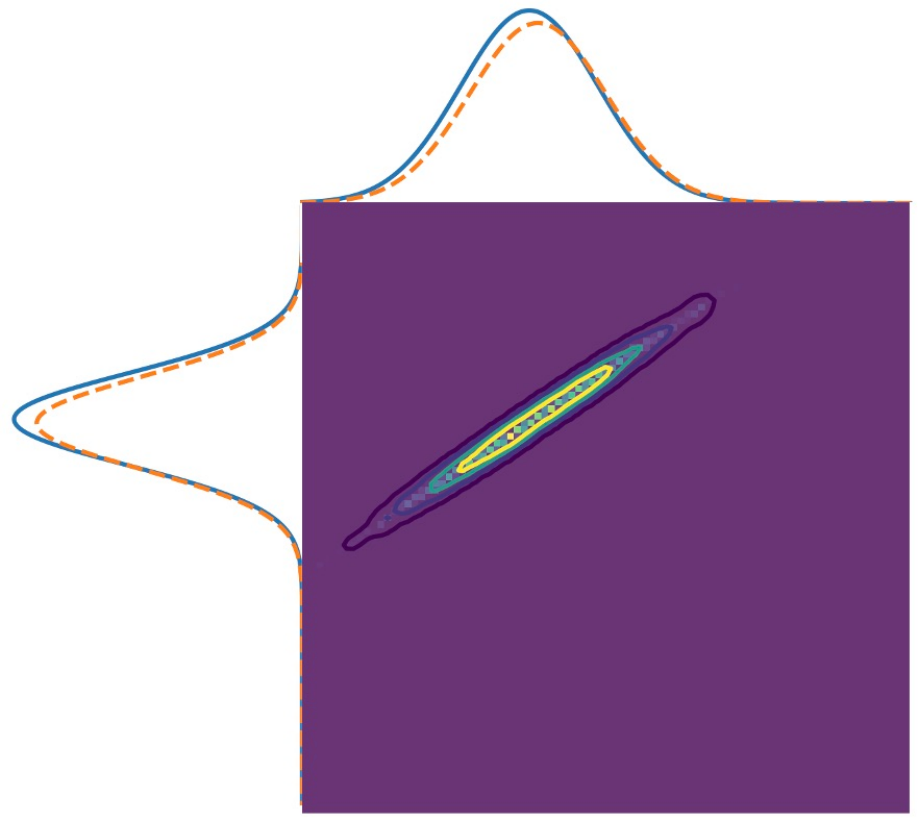}
  \hspace{5mm}
  \includegraphics[scale=0.25]{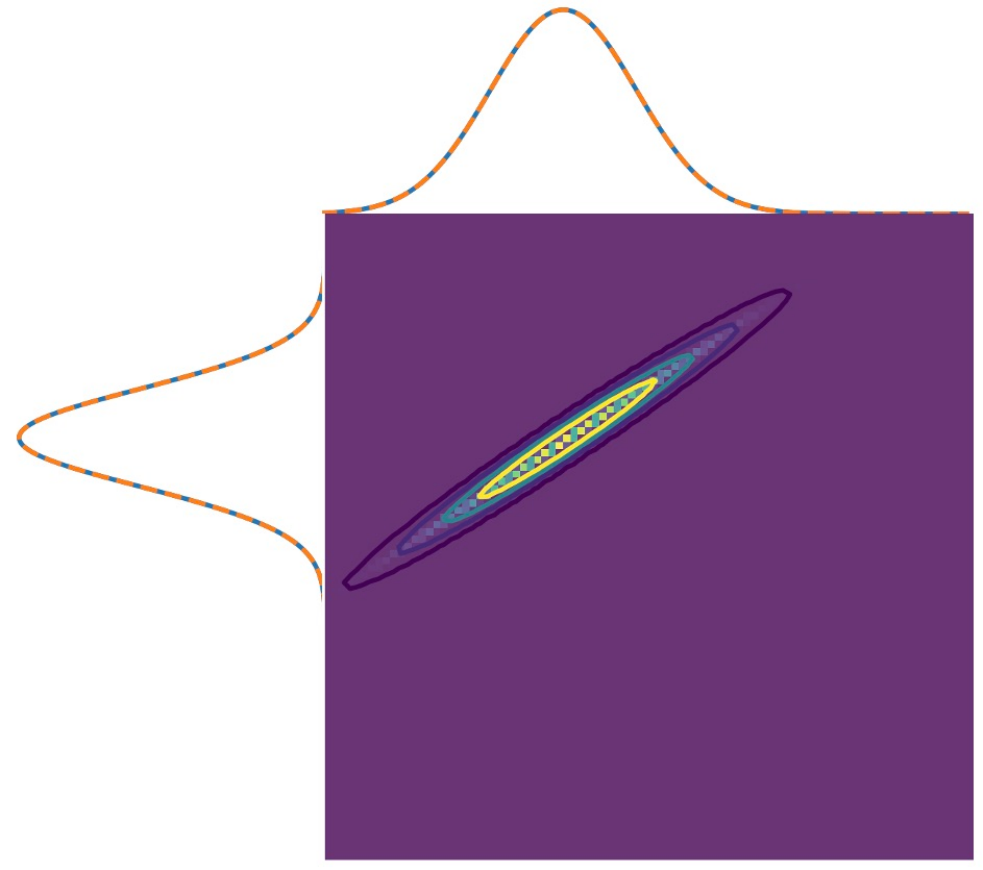}
  \caption{The optimal transport plans for the UOT problem with $\gamma=0.2$ (left) and $\gamma=30$ (middle), and for standard OT (right).}
\end{figure*}

\section{Global Solution Method to UDC}
\label{section:udc_solution}

In the previous section, we showed UOT problem---originally posed as a variational problem---can be transformed without loss of optimality into a finite-dimensional optimization over the first and second moments, and can be solved globally via convex optimization.
In this section, we establish an analogous result for the UDC problem \eqref{eq:udc}. 

\subsection{Restriction to Gaussian marginals}
\label{subsection:gaussian_affine_reduction}

Let the objective function of \eqref{eq:udc} be denoted by $J(\pi_1,\{U_k\})$.
Consider any feasible process of~\eqref{eq:udc} and let  $\{c\,p_k(x,u)\}$ be the corresponding measure, where each $p_k$ is a probability measure and $c>0$ denotes the mass. Namely, for given \(\pi_1,\,\{U_k\}\)
\begin{align*}
    c:= \int \rmd\pi_1,\quad
    p_k(\rmd x,\rmd u):= \frac{1}{c}U_k(\rmd u|x)\pi_t(\rmd x)
\end{align*}
Denote the first and second moments as follows.
\begin{align}\label{eq:moments_def}
    m_k &:= \bbE_{x_k\sim p_k}[x_k], \quad\Sigma_k := {\rm Cov}_{x_k\sim p_k}(x_k), \nonumber\\
    \bar u_k &:= \bbE_{u_k\sim p_k}[u_k], \quad \Sigma^u_k := {\rm Cov}_{u_k\sim p_k}(u_k),  \nonumber\\
    \Lambda_k &:= {\rm Cov}_{(x_k,u_k)\sim p_k}(u_k,x_k).
\end{align}

Using these moments, we can construct a Gaussian marginal and a Gaussian conditional control law that match the same first, second moments:
\begin{align}\label{eq:gaussian_affine_policy}
    \pi_1^{\rm G} &= c\,\calN(m_1,\Sigma_1), \nonumber\\
    U_k^{\rm G}(\cdot | x) &= \calN\!\Bigl(\bar u_k + \Lambda_k \Sigma_k^{-1}(x-m_k),\; \Sigma^u_k - \Lambda_k \Sigma_k^{-1}\Lambda_k^\top\Bigr).
\end{align}

\begin{thm}[Optimality of Gaussian]\label{thm:udc_gaussian_restriction}
    For any feasible solution to the UDC problem~\eqref{eq:udc} \(\pi_1,\,\{U_k\}\), define Gaussian control law as~\eqref{eq:gaussian_affine_policy}.
    Then it holds that 
    \[
    J\bigl(\pi_1,\{U_k\}\bigr)\ \ge\ J\bigl(\pi_1^{\rm G},\{U_k^{\rm G}\}\bigr).
    \]
\end{thm}

As a consequence, we may restrict attention to Gaussian state measures and Gaussian--affine feedback policies of the form
\[
u_k \sim K_k(x_k-m_k)+v_k+w_k,\qquad w_k\sim \calN(0,\Sigma^u_k).
\]
\subsection{Separation of the optimization}
\label{subsec:finite_dim_formulation}

Under the above restriction, the UDC problem is equivalently written as the following optimization over
$c$, the mean trajectory $\{m_k\}_{k}$, the covariance trajectory $\{\Sigma_k\}_{k}$, and the affine-feedback parameters $\{v_k,K_k,\Sigma^u_k\}_{k}$:
\begin{align}\label{eq:udc_reduced}
    &\min_{c,\{m_k,\Sigma_k,v_k,K_k,\Sigma^u_k\}} \,\,  c\Biggl[ \sum_{k=1}^{T-1}\Bigl(\|v_k\|^2 + \tr(K_k\Sigma_k K_k^\top) \nonumber\\
    & \qquad\qquad\qquad\qquad+ \tr(\Sigma^u_k)\Bigr) + \gamma M(m_1,m_T)  \nonumber\\
    & \qquad\qquad\qquad\qquad+ \gamma S(\Sigma_1,\Sigma_T) \Biggr] + \gamma \psi(c) \\
    &\qquad\,\, \text{s.t.}\,\,\, m_{k+1}=A m_k + B v_k,\nonumber\\
    & \qquad\qquad\,\,\Sigma_{k+1}=(A+BK_k)\Sigma_k(A+BK_k)^\top + B\Sigma^u_kB^\top \nonumber\\
    &\qquad\qquad\,\,k=1,\dots,T-1.\nonumber
\end{align}

Here, \(M, S\) are given by
\begin{align}
    &M(m_1,m_T)
    :=\frac12(m_1-m_\alpha)^\top \Sigma_\alpha^{-1}(m_1-m_\alpha) \nonumber\\
    &\qquad\qquad\qquad+\frac12(m_T-m_\beta)^\top \Sigma_\beta^{-1}(m_T-m_\beta),
    \label{eq:M_def}\\
    &S(\Sigma_1,\Sigma_T) :=\frac12\Bigl\{\tr(\Sigma_\alpha^{-1}\Sigma_1)+\tr(\Sigma_\beta^{-1}\Sigma_T)\nonumber\\
    &\qquad\qquad\qquad-\log|\Sigma_1|-\log|\Sigma_T|\Bigr\},\label{eq:S_def}
\end{align}
and $\psi(c)$ is the same as the UOT case~\eqref{eq:def-of-psi}.

As for UOT, although the problem \eqref{eq:udc_reduced} is not jointly convex, we can implement the  optimization over \(c\) and the one over the remaining variables separately. We introduce the following subproblem.
\begin{align}\label{eq:udc_subproblem}
    &\min_{\{m_k,\Sigma_k,v_k,K_k,\Sigma^u_k\}} \,\,   \sum_{k=1}^{T-1}\Bigl(\|v_k\|^2 + \tr(K_k\Sigma_k K_k^\top) \nonumber\\
    & \qquad\qquad\qquad\qquad+ \tr(\Sigma^u_k)\Bigr) + \gamma M(m_1,m_T)  \nonumber\\
    & \qquad\qquad\qquad\qquad+ \gamma S(\Sigma_1,\Sigma_T)  \\
    &\qquad\,\, \text{s.t.}\,\,\, m_{k+1}=A m_k + B v_k,\nonumber\\
    & \qquad\qquad\,\,\Sigma_{k+1}=(A+BK_k)\Sigma_k(A+BK_k)^\top + B\Sigma^u_kB^\top \nonumber\\
    &\qquad\qquad\,\,k=1,\dots,T-1.\nonumber
\end{align}

Problem~\eqref{eq:udc_subproblem} is not jointly convex in $(K_k,\Sigma_k,\Sigma_k^u)$.
Using the standard covariance steering reparametrization~\cite{Balci2022}, the covariance recursion becomes affine and the remaining constraint can be written as an LMI via the Schur complement. Hence, the problem reduces to an SDP:
\begin{align}\label{eq:sdp_subproblem}
    &\min_{\{m_k,\Sigma_k,v_k,S_k,Y_k\}} \,\,  \sum_{k=1}^{T-1}\Bigl(\|v_k\|^2 + \tr(Y_k)\Bigr) + \gamma M(m_1,m_T)  \nonumber\\
    &\qquad\qquad\qquad\qquad\,+ \gamma S(\Sigma_1,\Sigma_T)\\
    &\qquad\,\, \text{s.t.}\,\,\, m_{k+1}=A m_k + B v_k,\nonumber\\
    & \qquad\qquad\,\Sigma_{k+1} = A\Sigma_k A^\top + BS_kA^\top + AS_k^\top B^\top + BY_kB^\top \nonumber\\
    &\qquad\qquad\,k=1,\dots,T-1.\nonumber
\end{align}

As in the UOT case, we can solve the UDC problem~\eqref{eq:udc} globally by an two-stage procedure which is analogous to Algorithm~\ref{alg:uot}. 



\section{Numerical Simulation}\label{section:simmulation}

We illustrate the proposed UOT algorithm on a one-dimensional example with balanced references ($c_\alpha=c_\beta$). Figure~\ref{fig:optimal-plan} compares the optimal transport plans for $\gamma=0.2$ and $\gamma=30$ with the standard OT plan balanced case. As $\gamma$ increases, the solution approaches the
standard OT plan, and the optimal marginals $\pi_1^\ast,\pi_2^\ast$ become closer to the references
$\alpha,\beta$, as expected from the objective function~\eqref{eq:uot}. Also, Focusing on the first term in the UOT objective, transporting mass over short distances incurs a small cost, whereas transporting it over long distances is expensive. Hence, the transport cost associated with a transport plan is low near the diagonal and increases as one moves away from it. For small $\gamma$, it is therefore expected that the optimal solution prioritizes reducing the transport cost rather than matching the marginals to the reference measures. Consistent with this intuition, the optimal transport plan for $\gamma=0.2$ concentrates along the diagonal.

\section{Conclusion}
In this article, we study the UOT problem for Gaussian measures, which enables the comparison of distributions with different total masses, and we propose an algorithm that computes a global solution. We then formulate the UDC problem as a dynamical extension of UOT and derive an analogous global solution method. Future work includes extending the framework to non-Gaussian reference measures and to other divergence choices for quantifying discrepancies between measures. Applications to real-world dynamical systems are also left for future investigation.


\bibliography{reference}             
                                                   







\end{document}